\documentclass[a4paper,11pt]{article}

\usepackage[latin1]{inputenc}
\usepackage[T1]{fontenc}
\usepackage{textcomp}
\usepackage[english]{babel}
\usepackage[tbtags]{amsmath}
\usepackage{amssymb}
\usepackage{amsfonts}
\usepackage{amsthm}
\usepackage{latexsym}
\usepackage{array}
\usepackage{bbm}
\usepackage{multicol}
\usepackage{stmaryrd}

\newcommand{\m}[1]{\mathbbm{#1}}

\newcommand{\q}[1]{\mathcal{#1}}

\newcommand{\e}{\varepsilon}

\newcommand{\imp}{\Longrightarrow}

\newtheorem{thm}{\rm \bf{Theorem}}
\newtheorem*{thm*}{\rm \bf{Theorem}}
\newtheorem{prop}{\rm \bf{Proposition}}
\newtheorem{cor}{\rm \bf{Corollary}}
\newtheorem{lem}{\rm \bf{Lemma}}
\newtheorem{defi}{\rm \bf{Definition}}
\newtheorem{nb}{\emph{Remark}}

\makeatletter
\def\blfootnote{\xdef\@thefnmark{}\@footnotetext}
\makeatother

\title{Scattering below critical energy for the radial 4D Yang-Mills
  equation and for the 2D corotational wave map system}
\date{\today}
\author{Raphaël Côte\footnote{\'Ecole polytechnique} \and Carlos E. Kenig\footnote{University of Chicago} \and Frank Merle\footnote{Université de Cergy-Pontoise, Centre National de la Recherche Scientifique and Institut des Hautes \'Etudes Scientifiques}}

\begin{document}

\allowdisplaybreaks

\maketitle

\blfootnote{ R.C. and F.M. are supported in part by ANR grant ONDE NONLIN and C.E.K. is supported in part by NSF.}

\begin{abstract}
Given $g$ and $f=gg'$, we consider solutions to the following non linear wave equation~:
\[ \left\{ \begin{array}{l}
\displaystyle u_{tt} - u_{rr} - \frac 1 r u_r = - \frac{f(u)}{r^2}, \\
(u,u_t)|_{t =0} = (u_0,u_1). \end{array} \right.
\]
Under suitable assumptions on $g$, this equation admits non-constant stationary solutions~: we denote $Q$ one with least energy. We caracterize completely the behavior as time goes to $\pm \infty$ of solutions $(u,u_t)$ corresponding to data with energy less than or equal to the energy of $Q$~: either it is $(Q,0)$ up to scaling, or it scatters in the energy space.

Our results include the cases of the 2 dimensional corotational wave map 
system, with target $\m S^2$, in the critical energy space, as well as the 4 
dimensional, radially symmetric Yang-Mills fields on Minkowski space, in the critical energy space.
\end{abstract}

\section{Introduction}

In this paper we study the asymptotic behavior of solutions to a class of 
non-linear wave equations in $\m R \times \m R$, with data in the natural 
energy space. The equations covered by our results include the 2 dimensional 
corotational wave map system, with target $\m S^2$, in the critical energy 
space, as well as the 4 dimensional, radially symmetric Yang-Mills fields on
Minkowski space, in the critical energy space.

The equations under consideration admit non-constant solutions
that are independent of time, of minimal energy, the so-called harmonic maps 
$Q$ (see \cite{C05b} and the discussion below). It is known, from the work of 
Struwe \cite{Str03}, that if the data has energy smaller than or equal to the
energy of $Q$, then the corresponding solution exists globally in time (see 
Proposition \ref{gwp} below). (A recent result \cite{KST06} shows that large energy data may lead to a finite time blow up solution for the 
2 dimensional corotational wave map system, with target $\m S^2$ -- see also \cite{RS06}). In 
this paper, we show that, for this class of solutions, an alternative holds~: 
either the data is $(Q,0)$ (or $(-Q,0)$ if $-Q$ is also a harmonic map), 
modulo the natural symmetries of the problem, and the solution is independent 
of time, or a (suitable) space-time norm is finite, which results in the 
scattering at times $\pm \infty$. Thus the asymptotic behavior as $t \to \pm 
\infty$ for solutions of energy smaller than or equal to that of $Q$, is 
completely described. Because of the existence of $Q$, the result is clearly 
sharp. 

The result is inspired by the recent works \cite{KM06a, KM06b} of the last 
two authors, who developed a method to attack such problems, reducing them, 
by a concentration-compactness approach, to a rigidity theorem. An important 
element in the proof of the rigidity theorem in  \cite{KM06a, KM06b} is the 
use of a virial identity. This is also the case in this work, where the virial
identity we use in the proof of Lemma \ref{viriel} is very close to the one 
used in Lemma 5.4 of \cite{KM06b}. Lemma \ref{viriel} in turn follows from Lemma \ref{varlem}, which has its origin in the work of the first author \cite{C05b}. The concentration-compactness approach we use here is the same as the one in \cite{KM06b}, with an important proviso. The results in \cite{KM06b} are established for dimension $N=3,4,5$, while here, in order to include the case of radial Yang-Mills in $\m R^4$, we need to deal with a case similar to $N=6$~; it also establishes the result in \cite{KM06b} for $N=6$. This is carried out in Theorem \ref{lwp} below. 

It is conjectured that similar results will hold without the restriction to data with symmetry (for wave maps or Yang-Mills fields). These are extremely challenging problems for future research.

We now turn to a more detailed description of our results. Let $g : \m R 
\to \m R$ be $C^3$ such that $g(0) =0$, $g'(0) =k
\in \m N^*$, denote $f=gg'$, and $N$ be the
surface of revolution with polar coordinates $(\rho,\theta) \in [0,
\infty) \times \m{S}^1$, and metric $ds^2 = d\rho^2 + g^2(\rho) d\theta^2$
(hence $N$ is fully determined by $g$).

We consider $u$, an equivariant wave map in dimension 2 with target $N$, or a radial solution to the critical Yang-Mills equations in dimension 4, that is, a solution to the following problem (see \cite{GWE98} for the derivation of the equation). 
\begin{equation} \label{wm} \left\{ \begin{array}{l}
\displaystyle u_{tt} - u_{rr} - \frac 1 r u_r = - \frac{f(u)}{r^2}, \\
(u,u_t)|_{t =0} = (u_0,u_1). \end{array} \right.
\end{equation}
At least formally, the energy is conserved by such wave maps~:
\[ E(u,u_t) = \int \left( u_t^2 + u_r^2 + \frac{g^2(u)}{r^2} \right)
rdr = E(u_0, u_1). \]
Shatah and Tahvildar-Zadeh \cite{STZ94} proved that (\ref{wm}) is locally
well posed in the energy space 
\[ \q H  \times L^2 = \{ (u_0, u_1) | E(u_0,u_1) < \infty. \}. \]
For such wave maps, energy is preserved.

>From Struwe \cite{Str03} we have the following dichotomy regarding long
time existence of solutions to (\ref{wm}), depending on the geometry
of the target manifold $N$, and thus on $g$~:
\begin{itemize}
\item[$\bullet$] If $g(\rho) >0$ for all $\rho >0$ (and $\int_0^\infty g(\rho) d\rho =\infty$, to prevent a sphere at infinity), then any finite energy wave map is global in time.
\item[$\bullet$] Otherwise there exists a non-constant harmonic map $Q$,
and one may have blow up (cf. \cite{RS06, KST06}).
\end{itemize}
Our goal in this paper is to study the latter case, and to describe the dynamics of equivariant wave maps and of radial solutions to the critical Yang-Mills equations in dimension 4,  with energy smaller or equal to $E(Q)$. 

\subsection{Statement of the result}

{\bfseries Notations and Assumptions :}

\bigskip

Denote by $v = W(t)(u_0,u_1)$ the solution to
\begin{equation} \label{linop} \left\{ 
\begin{array}{l}
\displaystyle u_{tt} - u_{rr} - \frac{1}{r} u_r - \frac{k^2}{r^2} u =0, \\
(u,u_t)|_{t=0} = (u_0,u_1).
\end{array} 
\right. \end{equation}
$W(t)$ is the linear operator associated with the wave equation with a
quadratic potential.

For a single function $u$, we use $E(u)$ for
$E(u,0)$, with a slight abuse of notation, and we also use
\[ E_a^b(u) = \int_a^b \left( u_r^2 + \frac{g^2(u)}{r^2} \right) rdr. \]
To avoid degeneracy (existence of infinitely small spheres), we assume that the set of points where $g$ vanishes is  discrete. Denote $G(\rho) = \int_0^\rho |g|$. $G$ is an increasing function. We make the following assumptions on $g$ (that is on $N$, the wave map target)~:
\begin{enumerate}
\item[(A1)] $g$ vanishes at some point other than 0, and we denote $C^* >0$ the 
smallest positive real satisfying $g(C^*) =0$.
\item[(A2)] $g'(0) = k \in \{1,2\}$ and if $k=1$,  we also have $g''(0) =0$.  
\item[(A3)] $g'(-\rho) \ge g'(\rho)$ for $\rho \in [0,C^*]$ and $g'(\rho) \ge 0$ for all $\rho \in [0,D^*]$, where we denote by $D^*$ the point in $[0, C^*]$ such that $G(D^*) = G(C^*)/2$.
\end{enumerate}

The first assumption is a necessary and sufficient condition on $g$ for the
existence of stationary solutions to (\ref{wm}), that is, non-constant
harmonic maps. Hence denote $Q \in \q H$ the solution to $rQ_r = g(Q)$,
with $Q(0)=0$, $Q(\infty) = C^*$ and $Q(1) = C^*/2$, so that $(Q,0)$ is a
stationary wave map (see \cite{C05b} for more details). Note that
\[ E(Q) = 2 G(C^*). \]

The second assumption is a technical one~: the restriction on the range of $k$ should be removable using harmonic analysis. Recall that $k \in \m N^*$, and   for equivariant wave maps, one usually assumes $g$ odd. 
To remain at a lower level of technicality, we stick to the two assumptions in (A2) which encompass the cases of greater interest (see below).

The first part of third assumption is a way to ensure that $Q$ is a non-constant harmonic map (with $Q(0)=0$) with least energy. The second part arises crucially in the proof of some positivity estimates. This assumption could be somehow relaxed, but as such encompasses the two cases below, avoiding technicalities which are beside the point. We conjecture that this assumption is removable.

These assumptions  encompass 
\begin{itemize}
\item[$\bullet$] corotational
equivariant wave maps to the sphere $\m S^2$ in energy critical dimension $n=2$
($g(u)=\sin u$, $f(u) = \sin(2u)/2)$, $k=1$ -- we refer to \cite{GWE98} for
more details).
\item[$\bullet$] the critical (4-dimensional) radial Yang-Mills
equation ($f(u) = 2u (1-u^2)$, $g(u)= (1-u^2)$, notice that to enter
our setting we should consider $\tilde g(u) = g(u-1) = u(2-u)$, $k=2$ -- we
refer to \cite{CSTZ98} for more details).
 \end{itemize}

Recall that if $u \in \q H$, then $u$ has finite limits at
$r \to 0$ and $r \to \infty$, which are zeroes of $g$~: we denote them by
$u(0)$ and $u(\infty)$ (see \cite[Lemma 1]{C05b}). We can now introduce
\begin{equation} \label{vdelta}
\q V(\delta) = \{ (u_0, u_1) \in \q H \times L^2 | E(u_0, u_1) <
E(Q)+ \delta, \ u_0(0) = u_0(\infty) = 0 \}.
\end{equation}
Denote $H= \left\{ u | \| u \|_H^2 = \int \left( {u}_r^2 +
\frac{u^2}{r^2} \right) rdr < \infty \right\}$.
As we shall see below (Lemma \ref{bounds}), for $\delta \le E(Q)$,
$\q V(\delta)$ is naturally endowed with the Hilbert norm
\begin{equation}
\| (u_0,u_1) \|_{H \times L^2}^2 = \| u_0 \|_{H}^2 + \| u_1 \|_{L^2}^2 = 
\int \left( u_1^2 + {u_0}_r^2 + \frac{u_0^2}{r^2} \right) rdr.
\end{equation}
Finally, for $I$ an interval of time, introduce the Strichartz space $S(I)
= L_{t \in I}^{\frac{2k+3}{k}}(dt) L^{\frac{2k+3}{k}} (r^{-2} 
dr)$ and
\[ \| u \|_{S(I)} = \| u \|_{L_{t\in I}^{2+3/k} (dt) L_r^{2+3/k}
  (r^{-2}dr)}. \]
Notice that $S(I)$ is simply the Strichartz space
$L^{2+3/k}_{t,x}$ adapted to the energy critical wave equation in
dimension $2k+2$ (see \cite{KM06b}), under the conjugation by the map $u \mapsto u/r^k$. This space appears naturally, see Section 3 for further details.

\begin{thm} \label{th1}
Assume $k=1$ or $k=2$, and $g$ satisfies {\rm(A1)}, {\rm(A2)} and {\rm(A3)}.
There exists $\delta = \delta(g)> 0$ such that the following holds.
Let $(u_0,u_1) \in \q V(\delta)$ and denote by $u(t)$ the corresponding wave
map. Then $u(t)$ is global in time, and scatters, in the sense that $\| u
\|_{S(\m R)} < \infty$. As a consequence, there exist 
$(u^{\pm}_0,u^\pm_1) \in H \times L^2$ such that
\[ \| u(t) - W(t) (u^{\pm}_0,u^\pm_1) \|_{H\times L^2} \to 0 \quad \text{as}
\quad t \to \pm \infty. \]
\end{thm}

As a direct consequence, we have the following

\begin{cor} \label{cor1}
Let $(u_0,u_1)$ be such that $E(u_0, u_1) \le E(Q,0)$, and denote by $u(t)$ the corresponding wave map. Then $u(t)$ is global and we have the following dichotomy~:
\begin{itemize}
\item[$\bullet$] If $u_0 = Q$ (or  $u_0 = -Q$
if $-Q$ is a harmonic map) up to scaling, then $u(t)$ is a constant harmonic map ($u_t(t)=0$).
\item[$\bullet$] Otherwise $u(t)$ scatters, in the sense that there exist 
$(u^{\pm}_0,u^{\pm}_1) \in H \times L^2$
such that
\[ \| u(t) - W(t) (u^{\pm}_0,u^{\pm}_1) \|_{H \times L^2} \to 0 \quad \text{as}
\quad t \to \pm \infty. \]
\end{itemize}
\end{cor}

\begin{nb}
The fact that $u(t)$ is global in time is a direct corollary of
\cite{Str03} (in fact one has global well posedness in $\q V(E(Q))$ as
recalled in Proposition \ref{gwp}). The new point in our result is linear
scattering. 
\end{nb}

\begin{nb}
We conjecture that $\delta = E(Q)$. The only point missing for this is
to improve Lemma \ref{varlem} to $\delta = E(Q)$.
\end{nb}

\begin{nb}
This result corresponds to what is expected in a ``focusing''
setting. Similarly, there is a defocusing setting, in the case $g(\rho) >0$
for $\rho >0$. Arguing in the same way as in Theorem \ref{th1}, we can
prove that if $g$ satifies (A2), (A3) and $g'(\rho) \ge 0$ for all $\rho \in \m R$, then any wave map is
global and scatters in the sense of Theorem \ref{th1}. Again, we conjecture that
the correct assumptions for this result are $g(\rho) > 0$ for $\rho >0$ and
$G(\rho) \to \pm \infty$ as $\rho \to \pm \infty$ (to prevent a sphere at
infinity).
\end{nb}

\section{Variational results and global well posedness in $\q V(E(Q)) $}

First recall the  pointwise bound derived from the energy
\begin{equation} \label{ptw}
\forall r,r' \in \m R^+, \quad |G(u(r)) - G(u(r'))| \le \frac 1 2
E_r^{r'}(u),
\end{equation}
with equality at points $r,r'$ if an only if there exist $\lambda >0$ and
$\e \in \{ -1, 1\}$ such that
\[ \forall \rho \in [r,r'], \quad u(\rho) = \e Q(\lambda \rho). \]
(See \cite[Proposition 1]{C05b}.)

\begin{lem}[$\q V(\delta)$ is stable through the wave map flow] \label{V(delta)stable}
If $u \in \q H$, $u$ is continuous and has limits at $0$ and $\infty$
which are points where $g$ vanishes~: we denote them $u(0)$ and $u(\infty)$.
Furthermore if $u(t)$ is a finite energy wave
map defined on some interval $I$ containing 0, then for all $t \in I$,
\[ \forall t \in I, \quad u(t,0) = u(0,0) \quad \text{and} \quad
u(t,\infty) = u(0,\infty). \]  
In particular, for all $\delta \ge 0$, $\q V(\delta)$ is preserved under
the wave map flow.
\end{lem}

\begin{proof}
The properties of $u$ are well known : see \cite{GWE98} or
\cite{C05b}. Let us prove that the $u(t,0)$ is constant in time by a
continuity argument. 

For all $y$ such that $g(y) =0$, denote $I_y = \{ t \in I | u(t,0) = y \}$.
Let $t \in I$.

As $g$ vanishes on a discrete set, denote $\e>0$ such that if $g(\rho) =
0$,  $|G(\rho) - G(u(t,0))| \ge 2\e$. Since $u$ is defined in $I$, it does
not concentrate energy in a neighbourhood of $(t,0)$ : there exists $\delta_0
, \delta_1>0$ such that
\[ \forall \tau \in [t-\delta_0,t+\delta_0], \quad E_0^{\delta_1}(u(\tau)) \le
\e. \]
>From this and the pointwise bound, we deduce
\[ \forall \tau \in [t-\delta_0,t+\delta_0], \forall r \in [0,\delta_1],
\quad |G(u(\tau),0) - G(u(\tau,r)| \le \e/2. \]
Now compute for $t' \in [t-\delta_0,t+\delta_0]$~:
\begin{align*}
\left| \int_0^{\delta_1} G(u)(t,\rho) d\rho -
\int_0^{\delta_1} G(u)(t',\rho) d\rho \right| 
& \le \int_0^{\delta_1} \int_{t}^{t'} g(u(\tau,\rho) |u_t(\tau,\rho)|
d\tau d\rho \\
& \le \frac 1 2 \int_t^{t'} E(u) d\tau \le \frac 1 2 E(u) |t-t'|.
\end{align*}
Suppose $t'$ is such that $u(t,0) \ne u(t',0)$, and then $ |G(u)(t,0) -
G(u)(t',0)| \ge 2\e$. Then
\begin{align*}
\lefteqn{ \left| \int_0^{\delta_1} G(u)(t,\rho) d\rho - \int_0^{\delta_1}
G(u)(t',\rho) d\rho \right| } \\
& \ge \left| \int_0^{\delta_1} ( (G(u)(t,\rho) -
G(u)(t,0))  + (G(u)(t,0) - G(u)(t',0)) + G(u)(t',0) - G(u)(t',\rho) ) ) 
d\rho \right| \\
& \ge \delta_1 (2\e -\e/2 - \e/2) \ge \delta_1 \e.
\end{align*}
We just proved that
\[ \frac 1 2 E(u) |t'-t| \ge \e \delta_1. \]
This means that $I_{u(t,0)}$ is open in $I$. In the same way, $I
\setminus I_{u(t,0)} = \bigcup_{y, \ y \ne u(t,0)} I_y$ is also open
in $I$, so that $I_{u(t,0)}$ is closed in $I$. As $I$ is
connected, $I = I_{u(t,0)}$.

Similarly, one can prove that $u(t,\infty)$ is constant in time. The rest
of the Lemma follows from conservation of energy.
\end{proof}

\begin{lem} \label{bounds}
There exists an increasing function $K: [0, 2E(Q)) \to [0, C^*)$, and a
decreasing function $\delta : [0, 2E(Q)) \to (0,1]$ such
that the following holds. For all $u \in \q H$ such that $E(u) < 2E(Q)$,
and $u(0) =u(\infty) =0$, one has the pointwise bound
\[ \forall r, \quad |u(r)| \le K(E(u)) < C^*. \]
Moreover, one has
\[ \delta(E(u)) \| u \|_H  \le E(u) \le \| g'\|_{L^\infty} \| u \|_H. \]
\end{lem}

\begin{proof}
>From the pointwise bound (\ref{ptw}), we have
\[ |G(u)(r)| = |G(u)(r)- G(u)(0)| \le \frac 1 2 E_0^r(u), \quad |G(u)(r)| \le
\frac 1 2 E_r^\infty(u). \]
So that $2|G(u)(r)| \le E(u) < 2E(Q)$. As $G$ is an increasing function on
$[-E(Q), E(Q)]$, and $|G(-\rho)| \ge G(\rho)$ for $\rho \in [0,C^*]$, we obtain 
\[ | u(r)| \le G^{-1}(E(u)/2) < G^{-1}(E(Q)) = C^*. \]
Then $K(\rho) = G^{-1}(\rho/2)$ fits.

We now turn to the second line. For the upper bound, notice that $g(0) =0$
so that $g^2(\rho) \le \| g' \|_{L^\infty}^2 \rho^2$, and $\| g' \|_{L^\infty}
\ge |g'(0)| \ge 1$.

For the lower bound, notice that as $|u| \le K(E(u)) < C^*$, then $g^2(u)
\ge \delta(E(u)) u^2$ for some positive continuous function $\delta :
(-C^*,C^*) \to (0,1]$ ($g(\rho)/\rho$ is a continuous positive function on
$(-C^*,C^*)$, $\delta(\rho) = \min(1,\inf \{ g(r)/r \ | \ |r| \le \rho
\})$).
\end{proof}

\begin{prop}[Struwe \cite{Str03}] \label{gwp}
Let $(u_0,u_1) \in \q V(E(Q))$. Then the corresponding wave map is global in
time, and satisfies the bound 
\[ \forall t, r \quad |u(t,r)| \le K(E(u_0,u_1)). \]
\end{prop}

\begin{proof}
Indeed suppose that $u$ blows-up, say at
time $T$. By Struwe \cite{Str03}, there exists a non-constant harmonic map $\tilde Q$, and two sequences $t_n \uparrow T$
and $\lambda(t_n)$ such that $\lambda(t_n) |T-t_n| \to \infty$ and
\[ u_n(t,r) = u \left( t_n + \frac{t}{\lambda(t_n)}, \frac{r}{\lambda(t_n)}
\right) \to \tilde Q(r) \quad H_{\mathrm{loc}} (]-1,1[_t \times \m R_r). \]
>From Lemma \ref{V(delta)stable}, one deduces $\tilde Q(0)=0$, and hence (with assumption (A3)) $|\tilde Q(\infty)| \ge C^*$.

However, as $(u,u_t) \in \q V(E(Q))$, from Lemma \ref{bounds}, $|u(t,r)|
\le K(E(u)) < C^*$ (uniformly in $t$).  Now $\{ r \ge 0| |\tilde Q(r)| \ge
(K(E(u)) + C^*)/2 \}$ is an interval of the form $[A_{E(u)}, \infty)$ ($\tilde Q$ is monotone) so that 
\[ \int_{t \in [-1/2,1/2]} \int_{[A_{E(u)}, A_{E(u)} +1]} |u_n(t,r) - \tilde Q(r)|^2 rdr dt \ge (C^* - K(E(u)))^2/4 \nrightarrow 0. \]
This is in contradiction with the $H_{\mathrm{loc}}$ convergence~: hence $u$
is global. 
\end{proof}

\section{Local Cauchy problem revisited}

Denote $\Delta =
\partial_{rr} + \frac{2k+1}{r} \partial_r = \frac{1}{r^{2k+1}} \partial_r
(r^{2k+1} \partial_r)$ the radial Laplacian in dimension $\m R^{2k+2}$ and
$U(t)$ the linear wave operator in $\m R^{2k+2}$~:
\[ U(t)(v_0,v_1) =  \cos(t \sqrt{-\Delta}) v_0 + \sqrt{-\Delta} \sin(t
\sqrt{-\Delta}) v_1. \]
Notice that 
\begin{equation} \label{utov}
W(t)(u_0,u_1) = r^k U(t)(u_0/r^k,u_1/r^k),
\end{equation}
as $v$ solves $v_{tt} - \Delta v =0$ if and only if $r^kv$ solves
(\ref{linop}). 

Given an interval $I$ of $\m R$, denote 
\begin{align} 
\| v \|_{N(I)} & = \| v(t,x) \|_{N(t \in I)}
\nonumber \\
& = \| v \|_{L^\infty_{t \in I} \dot H^1_x} + \| v
\|_{L^{\frac{2k+3}{k}}_{t \in I,x}} + \| v
\|_{L^{\frac{2(2k+3)}{2k+1}}_{t \in I} \dot W^{1/2,\frac{2(2k+3)}{2k+1}}_x}
+ \| v \|_{W^{1,\infty}_{t \in I} L^2_x},  
\end{align}
where the space variable $x$ belongs to $\m R^{2k+2}$. This norm appears in
the Strichartz estimate (Lemma \ref{strichartz}). 

\begin{thm} \label{lwp}
Assume $k=1$ or $2$. Problem (\ref{wm}) is locally well-posed in the space
$H$ in the sense that there exist two functions
$\delta_0,C : [0,\infty) \to (0,\infty)$ such that the following holds. Let
$(u_0,u_1) \in H\times L^2$ be such that $\| u_0,u_1 \|_{H \times L^2} \le
A$, and let $I$ be an open interval containing 0 such that
\[ \| W(t) (u_0,u_1) \|_{S(I)}  = \eta \le \delta_0(A). \] 
Then there exist a unique solution $u \in C(I,H) \cap S(I)$ to
Problem (\ref{wm}) and $\| u \|_{S(I)} \le C(A) \eta$, (and we also have
$\| u/r^k \|_{N(I)} \le C(A)$ and $E(u,u_t) = E(u_0,u_1)$).

As a consequence, if $u$ is such a solution defined on $I= \m R^+$,
satisfying $\| u \|_{S(\m R^+)} < \infty$, there exist $(u^{+}_0,u^+_1) \in
H \times L^2$ such that 
\[ \| u(t) - W(t) (u^+_0,u^+_1) \|_{H \times L^2} \to 0 \quad \text{as}
\quad t \to +\infty. \]
\end{thm}

\subsection{Preliminary lemmas}

Let us first recall some useful lemmas. We consider
$D^s =(-\Delta)^{s/2}$ the fractional derivative operator and
the homogeneous Sobolev space
\[ \dot W^{s,p} =  \dot W^{s,p}(\m R^n) = \left\{ \varphi \in \q S'(\m
R^n) \ \left| \  \| \varphi 
\|_{ \dot W^{s,p}} \stackrel{\text{def}}{=} \|  D^{s} \varphi \|_{L^p} <
\infty \right. \right\}. \]
For integer $s$, it is well known that $\| \cdot \|_{\dot W^{s,p}}$ is
equivalent to the Sobolev semi-norm~:
\[ \| \varphi \|_{\dot W^{s,p}} \sim \| \nabla^s \varphi \|_{L^p}. \]

\begin{lem}[Hardy-Sobolev embedding] \label{hardy}
Let $n \ge 3$, and $p,q,\alpha,\beta \ge 0$ be such that $1
\le q \le p \le \infty$, and $0 < (\beta-\alpha)q < n$. There exist $C=
C(n,p,q,\alpha,\beta)$ such that for all $\varphi$ radial in $\m R^n$,
\[ \| r^{\frac n q - \frac n p - \beta + \alpha} \varphi \|_{\dot
  W^{\alpha,p}} \le C \| \varphi \|_{\dot W^{\beta,q}}. \]
\end{lem}

\begin{proof}
Given $n,p,q$ and $\beta$, we show the estimate for $\alpha$ in the suitable range.

The case $\alpha = 0$ is the standard Hardy inequality in $L^p$ combined
with the Sobolev embedding (see \cite{STZ94} and the references therein - where the conditions $n \ge 3$, $1 \le q \le p \le \infty$ and $0< \beta< n$ are required). If $\alpha$ is an integer, we use the Sobolev semi-norm~: as
\[ \partial_r^\alpha (r^\gamma v) = \sum_{k=0}^\alpha c_k r^{\gamma-k}
\partial_r^{\alpha-\gamma} v, \]
the inequality follows from the case $\alpha =0$.

In the general case, let $\alpha = k + \theta$ for $k \in \m N$ and $\theta
\in ]0,1[$, and $\gamma = \frac n q - \frac n p - \beta + \alpha$ . We define $\ell$ so that $\beta = \ell + \theta$, hence $\frac
n q - \frac n p - \ell + k = \gamma$. We consider the operator 
$T : \varphi \mapsto D^{k} (r^\gamma D^{-\ell} \varphi)$~: $T$
maps $L^q$ to $L^p$ and $\dot W^{1,q}$ to $\dot W^{1,p}$ (integer case). By complex interpolation (see \cite{Ste70b}), $T$ maps $[L^q, \dot W^{1,q}]_{\theta} = \dot W^{\theta, q}$ to  $[L^p, \dot W^{1,p}]_{\theta} = \dot W^{\theta, p}$. This means that
\[ \| r^\gamma \varphi \|_{\dot W^{k+\theta,p}} \le C \| \varphi \|_{\dot W^{\ell + \theta,q}}, \]
which is what we needed to prove.
\end{proof}

\begin{lem} \label{HqH} If $v = u/r^k$, then   
\[ \frac 1 3 \int v_r^2 r^{2k+1} dr \le \int \left( u_r^2 + \frac{u^2}{r^2}
\right) rdr \le (k^2+1) \int v_r^2 r^{2k+1} dr . \]
\end{lem}

\begin{proof}
First notice that $v_r = -k u/r^{k+1} + u_r/r^k$, hence $v_r^2 \le (k^2+1)
(u^2/r^{2k+2} + u_r^2/r^{2k})$ and
\[ \int v_r^2 r^{2k+1} dr \le (k^2+1) \int \left( u_r^2 + \frac{u^2}{r^2}
\right) rdr. \] 
Then from the Hardy-Sobolev inequality in dimension $2k+2
\ge 3$ (optimal constant is $1/k^2$),  
\[ \int \frac{u^2}{r^2} rdr = \int \frac{v^2}{r^2} r^{2k+1} dr \le \frac 1
{k^2} \int v_r^2 r^{2k+1}dr. \]
As $u_r = r^kv_r + ku/r$, $u_r^2 \le 2 r^{2k} v_r^2 + 2k^2 u^2/r^2$ and
\[ \int \left( u_r^2 + \frac{u^2}{r^2} \right) rdr \le \left( 2 +
\frac 1 {k^2} \right) \int v_r^2 r^{2k+1}rdr. \qedhere \]
\end{proof}

\begin{lem}[Derivation rules] \label{diff}
Let $1 < p <\infty$, $0<\alpha <1$. Then
\begin{align*}
\| D^{\alpha} (\varphi \psi) \|_{L^p} & \le C \| \varphi \|_{L^{p_1}} \|
D^\alpha \psi \|_{L^{p_2}} + \| D^\alpha \varphi \|_{L^{p_3}} \| \psi
\|_{L^{p_4}}, \\ 
\| D^\alpha (h(\varphi)) \|_{L^p} &\le C \| h'(\varphi) \|_{L^{p_1}} \|
D^\alpha \varphi \|_{L^{p_2}}. \\
\| D^{\alpha} (h(\varphi) - h(\psi)) \|_{L^p} & \le C ( \| h'(\varphi)
\|_{L^{p_1}} + \| 
h'(\psi) \|_{L^{p_1}} ) \| D^\alpha (\varphi-\psi) \|_{L^{p_2}}  \\
& \quad + C (\| h''(\varphi) \|_{L^{r_1}} + h''(\psi) \|_{L^{r_1}} )( \|
D^\alpha \varphi
\|_{L^{r_2}} + \| D^\alpha \psi \|_{L^{r_2}} ) \| \varphi -\psi \|_{L^{r_3}}, 
\end{align*}
where $\frac 1 p = \frac 1 {p_1} + \frac 1 {p_2} = \frac 1 {p_3} + \frac 1
{p_4} = \frac 1 {r_1} + \frac 1 {r_2} + \frac 1 {r_3}$, and $1<p_2, p_3, r_1,
r_2, r_3<\infty$.
\end{lem}

\begin{proof}
See \cite[Theorem A.6 and A.8]{KPV93} with functions which do not depend on
times, \cite[Theorem A.7 and A.12]{KPV93} and \cite[Lemma 2.5]{KM06b}.
\end{proof}

>From now on, we work in dimension $2k+2$ (radial), and the underlying measure 
is $r^{2k+1} dr$ unless otherwise stated. In particular, notice that from
Lemma \ref{diff}, we have~:
\begin{equation} \label{startest}
\| D^{1/2} (\varphi \psi) \|_{L^{\frac{2(2k+3)}{2k+5}}} \le \| D^{1/2} \varphi 
\|_{L^{\frac{2(2k+3)}{2k+5}}} \| \psi \|_{L^\infty} + \| \varphi 
\|_{L^{\frac{4(2k^2+5k+3)}{4k^2+12k+7}}} \| D^{1/2} \psi \|_{L^{4(k+1)}}.
\end{equation}

Recall 
\[ w = \cos(t \sqrt{-\Delta}) v_0 + \frac{\sin(t
  \sqrt{-\Delta})}{\sqrt{-\Delta}} 
v_1 + \int_0^t \frac{\sin((t-s) \sqrt{-\Delta})}{\sqrt{-\Delta}} \chi(s) ds \]
solves the problem
\[ \left\{ \begin{array}{l}
w_{tt} - \Delta w = \chi, \\
(w,w_t)|_{t=0} = (v_0,v_1), \\
\end{array} \right. \]

\begin{lem}[Strichartz estimate] \label{strichartz}
Let $I$ be an interval. There exist a constant $C$ (not depending on $I$)
such that (in dimension $2k+2$), 
\begin{align*}
\| \cos(t \sqrt{-\Delta}) v_0 \|_{N(\m R)} & \le C \| v_0
\|_{\dot H^1_x}, \\ 
\| \frac{\sin(t \sqrt{-\Delta})}{\sqrt{-\Delta}} v_1 \|_{N(\m R)} & \le \| v_1
\|_{L^2_x}, \\ 
\| \int_0^t \frac{\sin((t-s) \sqrt{-\Delta})}{\sqrt{-\Delta}} \chi(s) ds
\|_{N(I)} & \le \| D^{1/2}_x \chi \|_{L^{\frac{2(2k+3)}{2k+5}}_{t
\in I} L^{\frac{2(2k+3)}{2k+5}}_x}.
\end{align*}
\end{lem}

\begin{proof}
This result is well-known~: see \cite{KM06b} and the references therein.
\end{proof}

\subsection{Proofs of Theorem \ref{lwp} in the case $k=1$ and $k=2$}

\begin{proof}[Proof of Theorem \ref{lwp}]
Denote  $v =u/r^k$. Then $v_r = u_r/r^k- ku/r^{k+1}$, $v_{rr} =
\frac{u_{rr}}{r^k} - \frac{2ku_r}{r^{k+1}} + k(k+1)\frac{u}{r^{k+2}}$, so that
\begin{equation} \label{veq}
\left\{ 
\begin{array}{l}
\displaystyle v_{tt} - v_{rr} - (2k+1)\frac{v_r}{r} = - \frac{f(r^kv) -
k^2r^kv}{(r^k v)^{1+2/k}} v^{1+2/k}, \\
(v,v_t)|_{t=0} = (v_0,v_1) = (u_0/r^k, u_1/r^k).
\end{array} \right.
\end{equation}
This is something like the energy critical wave equation in dimension $2k+2$.

Denote 
\[ h(\rho) = \frac{f(\rho)-k^2\rho}{\rho^{1+2/k}}. \] 
Assume that $h$, $h'$ and $h''$ are bounded on compact sets~: this is automatic if $g$ is $C^3$ and satifies (A2). Indeed if $k=2$, $1+2/k =2$ and it is a direct application of Taylor's expansion, and if $k=1$, $1+2/k =3$, and it suffices to notice additionally that $f''(0) = 3k g''(0) =0$).

Our assumptions on $(u_0,u_1)$ translate to :
\[ \| v_0 \|_{\dot H^1} + \| v_1 \|_{L^2} \le CA, \qquad \| U(t)(v_0,v_1) 
\|_{L^{2+3/k}_{t \in I} L^{2+3/k}_x} \le C \eta. \]
Consider the map $\Phi$~:
\[ \Phi : v \mapsto\cos(t \sqrt{-\Delta}) v_0 + \frac{\sin(t
\sqrt{-\Delta})}{\sqrt{-\Delta}} v_1 + \int_{0}^t \frac{\sin((t-s)
\sqrt{-\Delta})}{\sqrt{-\Delta}} (v^{1+2/k}(s) h(r^kv)(s)) ds, \] 
that is $\Phi(v)$ solves the (linear in $\Phi(v)$) equation
\begin{equation}
\left\{ 
\begin{array}{l}
\displaystyle \Phi(v)_{tt} - \Phi(v)_{rr} - (2k+1)\frac{\Phi(v)_r}{r} = 
- h(r^k v) v^{1+2/k}, \\
(v,v_t)|_{t=0} = (v_0,v_1) = (u_0/r^k, u_1/r^k).
\end{array} \right.
\end{equation}
We will find a fixed point for $\Phi$, related to smallness in the norm :
\[ \| v \|_{L^{2+3/k}_{t \in I,r}} \quad \text{and} \quad \|
D^{1/2} v \|_{L^{2(2k+3)/(2k+1)}_{t \in I,r}}. \]
The Strichartz estimate shows that we are to control $\| D^{1/2}_x (v^{1+2/k} 
h(r^{1+2/k}v)) \|_{L^{\frac{2(2k+3)}{2k+5}}_{t \in I,x}}$.
For convenience in the following, denote :
\[ p = \frac{4(k+2)(2k^2+5k+3)}{4k(k^2+12k+7)}. \]
Now, we use (\ref{startest}) together with Lemma \ref{hardy} and Lemma
\ref{diff}~:
\begin{align*}
\lefteqn{ \| D^{1/2} (v^{1+2/k} h(r^k v)) \|_{L^{\frac{2(2k+3)}{2k+5}}} } \\
& \le \| D^{1/2} (v^{1+2/k}) \|_{L^{\frac{2(2k+3)}{2k+5}}} \| h(r^k v) 
\|_{L^\infty} + \| v^{1+2/k} \|_{L^{\frac{4(2k^2+5k+3)}{4k^2+12k+7}}} \| 
D^{1/2} h(r^k v) \|_{L^{4(k+1)}} \\
& \le C \| v^{2/k} \|_{L^{k+3/2}} \| D^{1/2} v \|_{L^{\frac{2(2k+3)}{2k+1}}}  
\| h(r^k v) \|_{L^\infty}
 + C \| v \|_{L^p}^{1+2/k} 
\| h'(r^k v) \|_{L^\infty} \| r^{k} v \|_{\dot W^{1/2,4(k+1)}} \\ 
& \le  C \| v \|_{L^{2+3/k}}^{2/k} \| D^{1/2} v
\|_{L^{\frac{2(2k+3)}{2k+1}}} \| h(r^k v) \|_{L^\infty}
+ C  \| v \|_{L^p}^{1+2/k} 
\| h'(r^k v) \|_{L^\infty} \| v_r \|_{L^2}.
\end{align*}
>From interpolation of Lebesgue spaces and Hölder inequality,
\begin{align} 
\left\| \| v \|_{L^p}^{1+2/k}
\right\|_{L^{\frac{2(2k+3)}{2k+5}}_t} & =
\left\| \| v \|_{L^{\frac{4(k+2)(2k^2+5k+3)}{k(4k^2 + 12k+7)}}_r} 
\right\|_{L^{(1+2/k)(2(2k+3)/(2k+5)}_t}^{1+2/k} \nonumber \\
& \le \| v \|_{L^{2+3/k}_{t,r}}^{2/k} \| v \|_{L^{2(2k+3)/(2k+1)}_t
  L^{\frac{4(2k+3)(k+1)}{4k^2+4k-1}}_r} \nonumber \\ 
& \le \| v \|_{L^{2+3/k}_{t,r}}^{2/k} \| D^{1/2}_r v
\|_{L^{\frac{2(2k+3)}{2k+1}}_{t,r}} \quad \text{and}  \label{stleb1} \\
\left\| \| v \|_{L^{2+3/k}_r}^{2/k} \| D^{1/2} v
\|_{L^{\frac{2(2k+3)}{2k+1}}_r} \right\|_{L^{\frac{2(2k+3)}{2k+5}}_t} & \le
\left\| \| v \|_{L^{2+3/k}_r}^{2/k} \right\|_{L^{k+3/2}_t} \|
D^{1/2} v \|_{L^{\frac{2(2k+3)}{2k+1}}_{t,r}} \nonumber \\
& \le \| v \|_{L^{2+3/k}_{t,r}}^{2/k} \| D^{1/2}_r v
\|_{L^{\frac{2(2k+3)}{2k+1}}_{t,r}}. \label{stleb2}
\end{align}
Using again Lemma \ref{hardy} to show $\| r^k v \|_{L^\infty} \le C \| v_r
\|_{L^2}$, we hence get our main estimate, for some
increasing function $\omega$ ($\omega$ is a function of $h,h'$ and
essentially the constant in the Strichartz estimate, and does not depend on
$I$ or $v$)~:
\begin{equation} \label{nonlinest}
\| D^{1/2}_x (v^{1+2/k} h(r^kv)) \|_{L^{\frac{2(2k+3)}{2k+5}}_{t \in I,r}} 
\le \omega(\| v_r \|_{L^\infty_{t\in I} L^2_r}) \| v 
\|_{L^{2+3/k}_{t \in I,r}}^{2/k} \| D^{1/2} v
  \|_{L^{\frac{2(2k+3)}{2k+1}}_{t \in I,r}}.
\end{equation} 
We now turn to difference estimates. Using the same inequalities, we get~:
\begin{align*}
\lefteqn{ \| D^{1/2} (v^{1+2/k} h(r^kv) - w^{1+2/k} h(r^kw)) 
\|_{L^{\frac{2(2k+3)}{2k+5}}_r} } \\
& \le C  \| D^{1/2} \left( (v^{1+2/k}-w^{1+2/k}) h(r^kv) \right) 
\|_{L^{\frac{2(2k+3)}{2k+5}}} \\
& \quad + C \left\| D^{1/2} \left( r^k w^{1+2/k} (v-w) \int_0^1 h'(\theta r^k
(v-w)+   
r^k w) d\theta \right) \right\|_{L^{\frac{2(2k+3)}{2k+5}}} \\ & \le \|
D^{1/2}(v^{1+2/k}-w^{1+2/k}) \|_{L^{\frac{2(2k+3)}{2k+5}}} \| h(r^kv)
\|_{L^\infty} \\
& \quad + \| v^{1+2/k}-w^{1+2/k}
\|_{L^{\frac{4(2k^2+5k+3)}{4k^2+12k+7}}} \| D^{1/2} h(r^kv) \|_{L^{4(k+1)}}
\\ 
& \quad + \| D^{1/2} (r^k w^{1+2/k} (v-w)) \|_{L^{\frac{2(2k+3)}{2k+5}}}
\left\| \int_0^1 h'(\theta r^k (v-w)+ r^kw) d\theta \right\|_{L^\infty} \\ 
& \quad + \| r^k w^{1+2/k} (v-w) \|_{L^{\frac{4(2k^2+5k+3)}{4k^2+12k+7}}}
\left\| \int_0^1 D^{1/2} (h'(\theta r^k (v-w)+ r^k w)) d\theta
\right\|_{L^{4(k+1)}} \\  
& \le \| D^{1/2}(v^{1+2/k}-w^{1+2/k}) \|_{L^{\frac{2(2k+3)}{2k+5}}} \|
h(r^k v) \|_{L^\infty} \\
& \quad + \| v-w \|_{L^{p}} ( \| v \|_{L^p}^{2/k} + \| w
\|_{L^p}^{2/k} ) \| h'(r^kv) \|_{L^\infty} \| v_r \|_{L^2} \\
& \quad + \left( \|D^{1/2} (w^{2/k}(v-w)) \|_{L^{\frac{2(2k+3)}{2k+5}}} \|
r^k w 
\|_{L^\infty} + \| w^{2/k}(v-w) \|_{L^{\frac{4(2k^2+5k+3)}{4k^2+12k+7}}}
\| D^{1/2} (r^k w) \|_{L^{4(k+1)}} \right) \\
& \qquad \times \sup_{\theta \in [0,1]} \| h'(r^kv +
\theta r^k(w-v)) \|_{L^\infty} + \| r^k w \|_{L^\infty} \| w^{2/k} (v-w)
\|_{L^{\frac{4(2k^2+5k+3)}{4k^2+12k+7}}} \\
& \qquad \qquad \times \sup_{\theta \in [0,1]} \left( \| 
h''(\theta r^k (v-w)+ r^kw) \|_{L^\infty} \| D^{1/2} (r^k(\theta v +
(1-\theta)w)) \|_{L^{4(k+1)}} \right).
\end{align*}
Then we have as previously~:
\begin{align*}
\lefteqn{ \|D^{1/2} (w^{2/k}(v-w)) \|_{L^{\frac{2(2k+3)}{2k+5}}} } \\
& \qquad \le C \| D^{1/2}
(v-w) \|_{L^{\frac{(2(2k+3)}{2k+1}}} \| w^{2/k} \|_{L^{k+3/2}} + C \| v-w
\|_{L^{2+3/k}} \| D^{1/2} (w^{2/k}) \|_{L^{\frac{2(2k+3)}{5}}} \\
& \qquad \le C \| w \|_{L^{2+3/k}}^{2/k} \| D^{1/2} (v-w)
\|_{L^{\frac{(2(2k+3)}{2k+1}}} + \| w \|_{L^{2+3/k}}^{2/k-1} \| D^{1/2} w  
\|_{L^{\frac{(2(2k+3)}{2k+1}}} \| v-w \|_{L^{2+3/k}}, \\
& \| w^{2/k}(v-w) \|_{L^{\frac{4(2k^2+5k+3)}{4k^2+12k+7}}} \le \| w
\|_{L^p}^{2/k} \| v-w \|_{L^p}.
\end{align*}
Doing the computations in each case $k=1$ or $k=2$, we have that
\begin{align*}
\| D^{1/2}(v^3-w^3) \|_{L^{10/7}} & \le \| D^{1/2} v-w
  \|_{L^{10/3}} (\| v \|_{L^{5}}^2 + \| w^2 \|_{L^{5}}^2) \\
& \quad + \| v -w \|_{L^5} (\| D^{1/2} v \|_{L^{10/3}} + \| D^{1/2} w
  \|_{L^{10/3}})(\| v \|_{L^5} + \| w \|_{L^5}) \quad \text{and} \\
\| D^{1/2} (v^2 - w^2) \|_{L^{14/9}} & = \| D^{1/2} ((v- w)(v+w))
  \|_{L^{14/9}} \\
& \le C \| D^{1/2} (v-w) \|_{L^{14/5}} (\| v \|_{L^{7/2}} + \| w
  \|_{L^{7/2}}) \\
& \quad + C \| v-w \|_{L^{7/2}} (\| D^{1/2} v \|_{L^{14/5}} + \|
  D^{1/2} w  \|_{L^{14/5}}).
\end{align*}
so that in both cases
\begin{multline*}
\|D^{1/2} (v^{1+2/k} - w^{1+2/k}) \|_{L^{\frac{2(2k+3)}{2k+5}}} \\
\le C (\| v \|_{L^{2+3/k}} + \| w \|_{L^{2+3/k}})^{2/k-1} \left( (\| v
\|_{L^{2+3/k}} + \| w \|_{L^{2+3/k}})  \| D^{1/2} 
(v-w) \|_{L^{\frac{2(2k+3)}{2k+1}}} \right. \\
\left. + (\| D^{1/2} v
\|_{L^{\frac{2(2k+3)}{2k+1}}} + \| D^{1/2} w
\|_{L^{\frac{2(2k+3)}{2k+1}}} \right)  \| v-w \|_{L^{2+3/k}}).
\end{multline*}
Here, the assumption $k \le 2$ is crucially needed. Finally observe that
\[ |\theta v + (1-\theta)w| \le |v|+|w|, \quad |D^{1/2} (\theta v +
(1-\theta) w) \le |D^{1/2} v| + |D^{1/2} w|. \] 
We can now summarize these computations, and using (\ref{stleb1}) and 
(\ref{stleb2}), we obtain the space time difference estimate (up to a
change in the function $\omega$, which now depends on $h$, $h'$ and $h''$,
but not on $I$ or $v$)~:
\begin{multline*}
\| D^{1/2} (v^{1+2/k} h(r^kv) - w^{1+2/k} h(r^kw))
\|_{L^{\frac{2(2k+3)}{2k+5}}_{t \in I,r}} \le (\omega(\| v \|_{L^\infty_{t
\in I} \dot H^1_r}) + \omega(\| w \|_{L^\infty_{t \in I} \dot H^1_r})) \\
\times (\| v \|_{L^{2+3/k}_{t \in I,r}}^{2/k-1} + \| w \|_{L^{2+3/k}_{t \in
I,r}}^{2/k-1}) \left( (\| v \|_{L^{2+3/k}_{t \in I,r}} + \| w
\|_{L^{2+3/k}_{t \in I,r}}) \| D^{1/2} (v-w)
\|_{L^{\frac{2(2k+3)}{2k+1}}_{t \in I,r}} \right.  \\ 
\qquad \left. +(\| D^{1/2} v \|_{L^{\frac{2(2k+3)}{2k+1}}_{t \in I,r}} + \|
D^{1/2} w \|_{L^{\frac{2(2k+3)}{2k+1}}_{t \in I,r}}) \| v-w
\|_{L^{2+3/k}_{t \in I,r}} \right).
\end{multline*}
Given $a,b,A \in \m R^+$, $I$ a time interval, introduce
 \[ B(a,b,A,I) = \left\{ v | \,  \| v \|_{L^{2+3/k}_{t \in I,r}} \le a, \ \ 
\| D^{1/2} v \|_{L^{\frac{2(2k+3)}{2k+1}}_{t \in I,r}} \le b,  \ \ \| v
\|_{C(t\in I,  \dot H^1_r)} \le 2CA  \right\}. \]
Hence for $v \in B(a,b,A,I)$, we have
\begin{align*}
\| \Phi(v) \|_{L^{2+3/k}_{t \in I,r}} & \le \| U(t)(v_0,v_1) 
\|_{L^{2+3/k}_{t \in I,r}} + \omega(2CA) a^{2/k} b \\
\| D^{1/2} \Phi(v) \|_{L^{\frac{2(2k+3)}{2k+1}}_{t \in I,r}} & \le \| D^{1/2}
U(t)(v_0,v_1)  
\|_{L^{\frac{2(2k+3)}{2k+1}}_{t \in I,r}} + \omega(2CA) a^{2/k} b \\
\| \Phi(v) \|_{C(t \in I, \dot H^1)} & \le \| (v_0,v_1) \|_{\dot H^1 \times
  L^2} + \omega(2CA) a^{2/k} b, \\
\| \Phi(v) - \Phi(w) \|_{N(I)} & \le 2\omega(2CA) a^{2/k-1} b (
\| D^{1/2} (v-w) \|_{L^{\frac{2(2k+3)}{2k+1}}_{t \in I,r}} + \| v-w 
\|_{L^{2+3/k}_{t \in I,r}})
\end{align*}

\emph{Case $k=1$}

We compute $2+3/k=5$ and $\frac{2(2k+3)}{2k+1} =10/3$.

Given $A$, set $b=2CA$ and $\delta_0(A) = \min(1,1/C, \frac{1}{8 CA 
\omega(2CA)})$.
Then for $(v_0,v_1)$ such that $\| (v_0,v_1) \|_{\dot H^1 \times L^2} \le
A$ and 
$\| U(t)(v_0,v_1) \|_{L^5_{t \in I,r}} = \eta \le \delta_0(A)$, set $a =
2\eta$. 
Notice that the Strichartz estimate gives
\[ \| D^{1/2} U(t)(v_0,v_1) \|_{L^{10/3}_{t \in I,r}} \le CA. \]
Our relations now write (the main point is $2/k-1 = 1 >0$)~:
\begin{align*}
\| \Phi(v) \|_{L^{5}_{t \in I,r}} & \le \frac a 2  + \omega(2CA) (2\delta_0 
a)(2CA) \le a \\
\| D^{1/2} \Phi(v) \|_{L^{10/3}_{t \in I,r}} & \le CA + 
\omega(2CA) (2\delta_0 a) (2CA) \le 2CA \\
\| \Phi(v) \|_{C(t \in I, \dot H^1)} & \le A + \omega(2CA) (2\delta_0 a) (2CA) 
\le 2 A, \\
\| \Phi(v) - \Phi(w) \|_{N(I)} & \le \frac 1 2 (\| D^{1/2} (v-w) 
\|_{L^{10/3}_{t \in I,r}} + \| v-w  \|_{L^{5}_{t \in I,r}})
\end{align*} 
Hence $\Phi : B(a,2CA,A,I) \to B(a,2CA,A,I)$ is a well defined $1/2$-Lipschitz
map, so that $\Phi$ has a unique fixed point, which is our solution.

\bigskip

\emph{Case $k=2$}

We compute $2+3/k = 7/2$, $\frac{2(2k+3)}{2k+1} = 14/5$ and
 $\frac{2(2k+3)}{2k+5} = 14/9$.

In this case $2/k-1=0$, so that the procedure used in the case $k=1$ no
longer applies (it is the same problem as for the energy critical wave equation
in dimension 6).

However, we still have a solution on an interval $I$ where both quantities 
$\| U(t)(v_0,v_1) \|_{L^{7/2}_{t \in I,r}}$ and $\| D^{1/2} U(t)(v_0,v_1) 
\|_{L^{14/5}_{t \in I,r}}$ are small.

Indeed, given $A$, set $\delta_1(A) = \min(1,\frac 1 C, \frac{1}{8
\omega(2CA)})$. For $(v_0,v_1)$ such that $\| (v_0,v_1) \|_{\dot H^1 \times
  L^2}  
\le A$, $\| U(t)(v_0,v_1) \|_{L^{7/2}_{t \in I,r}} = \eta \le \delta_1(A)$, and
$\| D^{1/2} U(t)(v_0,v_1) \|_{L^{14/5}_{t \in I,r}} = \eta' \le
\delta_1(A)$, we set $a = 2 \eta$ and $b = 2 \eta'$. Then we have
\begin{align*}
\| \Phi(v) \|_{L^{7/2}_{t \in I,r}} & \le \frac a 2  + \omega(2CA)
a)(2\delta_0) \le a \\
\| D^{1/2} \Phi(v) \|_{L^{14/5}_{t \in I,r}} & \le \frac b 2 + 
\omega(2CA) (2\delta_1(A)) b \le b \\
\| \Phi(v) \|_{C(t \in I, \dot H^1)} & \le A + \omega(2CA) (2\delta_1(A))^2 
\le 2 A, \\
\| \Phi(v) - \Phi(w) \|_{N(I)} & \le \frac 1 2 (\| D^{1/2} (v-w) 
\|_{L^{14/5}_{t \in I,r}} + \| v-w  \|_{L^{7/2}_{t \in I,r}})
\end{align*}
Hence $\Phi~: B(a,b,A,I) \to B(a,b,A,I)$ has a unique fixed point. We just
proved the following

\emph{Claim :} Let $A >0$. There exist $\delta_1(A) >0$ such that for
$(v_0,v_1)$ with $\| (v_0,v_1) \|_{\dot H^1 \times L^2} \le A$, and $I$
such that  
\[ \| U(t)(v_0,v_1) \|_{L^{7/2}_{t \in I,r}} = \eta \le \delta_1(A), \quad
\text{and} \quad \| D^{1/2} U(t)(v_0,v_1) \|_{L^{14/5}_{t \in I,r}} = \eta' \le
\delta_1(A), \]
Then there exist a unique solution $v(t)$ to (\ref{veq}) satisfying
\[ \| (v,v_t) \|_{L^\infty_{t \in I} (\dot H^1 \times L^2)} \le 2A, \quad \| v \|_{L^{7/2}_{t
\in I, r}} \le 2\eta, \quad \| D^{1/2} v \|_{L^{14/5}_{t \in I, r}} \le
2\eta'. \]

\bigskip

Let us now do a small computation.

Given $h$, $n \in \m N$ and $0 = t_0 < t_1 < \ldots < t_n = T$ (with $T \in
(0,\infty]$), we have for $i = 0, \ldots, n$,
\begin{align}
\lefteqn{ \| \int_{0}^{t} \frac{\sin((t-s)
\sqrt{-\Delta})}{\sqrt{-\Delta}} \chi(s) ds \|_{N(t_i,t_{i+1})} } \nonumber \\
& \le \sum_{j=0}^{i-1} \| \int_{t_j}^{t_{j+1}}
\frac{\sin((t-s) \sqrt{-\Delta})}{\sqrt{-\Delta}} \chi(s) ds
\|_{N(t_i,t_{i+1})} +  \| \int_{t_i}^t \frac{\sin((t-s)
\sqrt{-\Delta})}{\sqrt{-\Delta}} \chi(s) ds \|_{N(t_i,t_{i+1})} \nonumber\\
& \le  \sum_{j=0}^{i-1} \| \int_{0}^{t}
\frac{\sin((t-s) \sqrt{-\Delta})}{\sqrt{-\Delta}} (\chi(s) \m 1_{s \in [t_j,
t_{j+1}]}) ds \|_{N(t_i,t_{i+1})}  \nonumber \\
& \qquad +  \| \int_{t_i}^t \frac{\sin((t-s)
\sqrt{-\Delta})}{\sqrt{-\Delta}} (\chi(s) \m 1_{s \in [t_i,t_{i+1}]}) ds
\|_{N(t_i,t_{i+1})}  \nonumber \\ 
& \le \sum_{j=0}^{i-1} \| \int_{0}^{t} \frac{\sin((t-s)
\sqrt{-\Delta})}{\sqrt{-\Delta}} (\chi(s) \m 1_{s \in [t_j, 
t_{j+1}]}) ds \|_{N(\m R)} \nonumber \\
& \quad +  \| \int_{0}^t \frac{\sin((t-s)
\sqrt{-\Delta})}{\sqrt{-\Delta}} (\chi(s) \m 1_{s \in [t_i,t_{i+1}]}) ds
\|_{N(\m R)} \nonumber  \\
& \le C \sum_{j=0}^{i} \| D^{1/2}_x \chi(s)  \m 1_{s \in [t_j, t_{j+1}]}
\|_{L^{14/9}_{s,x}} 
\le C \sum_{j=0}^{i} \| D^{1/2}_x \chi \|_{L^{14/9}_{t \in [t_j,t_j+1]}
L^{14/9}_x}  \label{decoupage}
\end{align}

Let us now complete the case $k=2$. Let $A >0$, define $n=n(A)$ such that
$n= n(A) = 1/(4 CA \omega(2CA))$, so that $2CA\omega(2CA)/n \le 1/2$ and
$\delta_0(A) = \delta_1(A)/2^{n+2}$ (recall $\delta_1(A) = \min(1,\frac 1 C,
\frac{1}{8CA \omega(2CA)})$).

Let $(v_0,v_1)$ be such that $\| v_0,v_1 \|_{\dot H^1 \times L^2} \le A$ and
for $I=(T_0,T_1)$ an interval (possibly with infinite endpoints), 
$\| U(t) (v_0,v_1) \|_{L^{7/2}_{t \in I,r}} = \eta \le \delta_0(A)$. 

>From the Strichartz estimate, we also have
\[ \| D^{1/2} U(t)(v_0,v_1) \|_{L^{14/5}_{t \in I,r}} \le CA. \]
>From $(v_0,v_1)$, we have a solution $v$ defined on a interval 
$\tilde I=[0,T)$. We choose $J =(T_0',T_1') \subset \tilde I$ to be maximal
such that 
\[ \| v \|_{L^{7/2}_{t \in J,r}} \le \delta_1(A), \quad \| D^{1/2} v
\|_{L^{14/5}_{t \in J,r}} \le 2CA, \quad \| v \|_{C(J,\dot H^1)} \le 2CA. \]
>From the claim, we can choose $J$ non empty.
Let $T_0' = t_0 < t_1 < \ldots t_n = T_1'$ be such that
\[ \forall i \in \llbracket 0, n-1 \rrbracket, \quad \| D^{1/2} v
\|_{L^{14,5}_{t \in [t_{i},t_{i+1}],r}} \le \frac{2CA}{n} \le\frac 1 2
  \frac{1}{\omega(2CA)}. \] 
>From (\ref{nonlinest}) and (\ref{decoupage}), we obtain
\begin{align*}
\| v \|_{N(J)} & \le CA + \omega(2CA) \| v \|_{L^{7/2}_{t \in J,r}}
\| v \|_{N(J)}, \\
\| v \|_{L^{7/2}_{t \in [t_i,t_{i+1}],r}} & \le \| U(t)
(v_0,v_1) \|_{L^{7/2}_{t \in [t_i,t_{i+1}],r}} + \omega(2CA) \sum_{j=0}^i \| v
\|_{L^{7/2}_{t\in [t_j,t_{j+1}],r}} \| D^{1/2} v \|_{L^{14/5}_{t\in
[t_j,t_{j+1}],r}}.
\end{align*}
Let us denote $a_i = \| v \|_{L^{7/2}_{t \in [t_i,t_{i+1}],r}}$ for $i \in
\llbracket 0, n-1 \rrbracket$. Then we have
\begin{gather} 
\| v \|_{N(J)} \le CA + \frac 1 4 \| D^{1/2} v \|_{L^{14/5}_{t \in I,r}}
\le 3/2 CA < 2CA, \label{Nest} \\ 
a_i \le \eta + \omega(2CA) \sum_{j=0}^i \frac{a_j}{2 \omega(2CA)} \quad
\text{or  equivalently} 
\quad a_i \le 2 \eta + \sum_{j=0}^{i-1} a_j. \nonumber
\end{gather}
By recurrence, we deduce that 
\[ a_i \le 2^{i+1} \eta. \]
In particular,
\begin{equation} \label{Lpest}
\| v \|_{L^{7/2}_{t \in J} L^{7/2}_r} = \sum_{i=0}^{n-1} a_i \le 2^{n+1}
\eta \le 2^{n+1} \delta_0(A) < \delta_1(A).
\end{equation}
Hence, from with (\ref{Nest}) and (\ref{Lpest}) and a standard continuity
argument, we deduce that $J=\tilde I = I$, 
$\| v \|_{N(I)} \le 2CA$ and $\| v \|_{L^{7/2}_{t \in I,r}} \le 2^{n+1} \eta
= c(A) \eta$.

Going back to $u$, we obtain the first part of Theorem \ref{lwp},
in both cases $k=1$ and $k=2$ (conservation of energy is clear from the
construction).

\bigskip

Let us now prove the consequence mentioned in Theorem \ref{lwp}. Given $u$, we
associate $v(t,r) = u(t,r)/r^k$~: $v$ is defined on $\m R^+$, and satisfies
(\ref{veq}).

If we denote $A= \| (u,u_t) \|_{L^\infty_t(H \times L^2)}$, then there
exist $T$ large enough such that $\| u \|_{S([T,\infty))} \le
\delta_0(A)$. From the previous part, we have that
\[ \| v \|_{N[T,\infty)} \le 2CA, \quad \| v \|_{L^{2+3/k}_{t \in
[T,\infty),r}} \le  \delta_0(A). \]
Denote $\nu(t) = U(-t) v(t)$. Then 
\[ \nu(t) - \nu(s) = \int_s^t U(-\tau) v^{1+2/k}(\tau) h(r^kv)(\tau) d\tau. \]
Hence, for $t \ge s \ge T$, from the Strichartz estimate and (\ref{nonlinest}),
we have
\begin{align*}
\| \nu(t) - \nu(s) \|_{\dot H^1} + \| \nu_t(t) - \nu_t(s) \|_{L^2}
& \le \| \nu(\tau) - \nu (s) \|_{N(\tau \in [s,t])} \\
& \le \|  v^{1+2/k}(\tau) h(r^kv)(\tau) \|_{L^{\frac{2(2k+1)}{2k+5}}_{\tau
\in [s,t], r}} \\
& \le \omega(2CA) \| v \|_{L^{2+3/k}_{\tau \in [s,t],r}}^{2/k} (2CA) \to 0
\quad \text{as} \quad s,t \to + \infty.
\end{align*}
This means that $(\nu(t), \nu_t(t))$ is a Cauchy sequence in $\dot H^1 \times
L^2$, hence converges to some $(v^+,v_t^+) \in \dot H^1 \times L^2$.  

Going back to $u$, using Lemma \ref{HqH} and remark (\ref{utov}), we
obtain the second part of Theorem \ref{lwp}.
\end{proof}

\section{Rigidity property}

Recall that $g$ is such that $g(0) =0$, $g'(0) = k \in \m N^*$, 
with $C^*$ the smallest positive real such that $g(C^*) =0$, $f = g'g$
and $G(\rho) = \int_0^\rho |g|(\rho') d\rho'$~;
$D^* \in [0,C^*]$ is such that $G(D^*) = G(C^*)/2$.

Introduce the energy density $e(u,v) = v^2+ u_r^2 + \frac{g^2(u)}{r^2}$ and
$p(u) = u_r^2 + \frac{g^2(u)}{r^2}$. Denote
\[ E(u,v) = \int e(u,v) rdr, \quad E_a^b(u,v) = \int_a^b e(u,v) rdr, \]
and similarly for a single function $u$
\[ E(u) = \int p(u) rdr, \qquad E_a^b(u) = \int_a^b p(u) rdr. \]
We will also need the function $d(\rho) = \rho f(\rho)$, which is linked to
the virial identity, and
\[ F(u) = \int \left( u_r^2 + \frac{d(u)}{r^2} \right) rdr. \]
The following variational Lemma is at the heart of the rigidity
theorem. Here is the only point where we use assumption (A3), which ensures 
that $g'(\rho) \ge 0$ for $\rho \in [-D^*,D^*]$.

\begin{lem} \label{varlem}
There exist $c>0$ and $\delta \in (0,E(Q))$ such that for all $u$ such that $(u,0) \in \q V(\delta)$, we have 
\[ c E(u) \le F(u) \le \frac 1 c  E(u). \]
\end{lem}

\begin{proof}
Fix $\delta < E(Q)$. $g^2(u) \ge \omega(\delta) u^2$ for some function
$\omega : [0,E(Q)) \to \m R^+_*$, and $|d(x)| \le \| g'
\|_{L^\infty}^2 x^2$, so that 
\[ F(u) \le \left( 1 + \frac{\| g' \|_{L^\infty}^2}{\omega(\delta)} \right)
E(u), \]
which is the upper bound.

For the lower bound, we need assumption (A3) on $g$. Hence on $[-D^*,D^*]$, $d(x) \ge 0$, and on $[0,D^*]$, $d(-x) \ge d(x)$. Denote $A = \int_0^{D^*}
\sqrt{d(x)} dx >0$. One easily sees that for $v : [a,b] \to [-D^*,D^*]$
such that $v(a) = 0$, $|v(b)| = D^*$ then
\[ \int_a^b \left( v_r^2 + \frac{d(v)}{r^2} \right) rdr \ge 2 \int_a^b |v_r
\sqrt{d(v(r))}| dr \ge 2 \int_0^{D^*} \sqrt{d(x)} dx = 2A. \]
In the same way,
\[ \int_a^b  \left( v_r^2 + \frac{g^2(v)}{r^2} \right) rdr \ge 2G(D^*) =
G(C^*). \]
Let $\delta >0$ to  be determined later and $u$ be such that $(u,0) \in \q V(\delta)$. 
Recall that $\| u \|_{L^\infty} \le K(E(Q)+\delta) < C^*$ (Lemma \ref{bounds}), and hence $g(u) \ge
\omega(E(Q)+\delta) |u|$.

Assume first $\| u \|_{L^\infty} > D^*$. Then let $A_1$, $A_2$ such that $u
\in [-D^*, D^*]$ on both intervals $[0,A_1]$ and $[A_2,\infty)$ and $|u(A_1)| =
|u(A_2)| = D^*$. Then
\begin{align*}
\int \left( u_r^2 + \frac{d(u)}{r^2} \right) rdr & = \int_0^{A_1} +
\int_{A_1}^{A_2} + \int_{A_2}^\infty \ge 4A + \int_{A_1}^{A_2} \left( u_r^2 +
\frac{d(u)}{r^2} \right) rdr
\end{align*}
Doing the same with the energy density, one gets 
\[ \int_0^{A_1} \left( u_r^2 + \frac{g^2(u)}{r^2} \right) rdr +
\int_{A_2}^{\infty} \left( u_r^2 + \frac{g^2(u)}{r^2} \right) rdr \ge 4
G(D^*) = 2 G(C^*) = E(Q). \]
Hence $E_{A_1}^{A_2}(u) <\delta$. Now, we have
\[ |d(u)| = |u||g'(u)||g(u)| \le \| g' \|_{L^\infty} |u| g(u) \le
\frac{\| g' \|_{L^\infty}}{\omega(E(Q)+\delta)} g^2(u), \]
so that
\begin{align*}
\int_{A_1}^{A_2} \left( u_r^2 +
\frac{d(u)}{r^2} \right) rdr \ge \int_{A_1}^{A_2} \left( u_r^2 -
\frac{\| g' \|_{L^\infty}}{\omega(E(Q)+\delta)} \frac{g^2(u)}{r^2} \right)
rdr \ge - \frac{\| g'\|_{L^\infty}}{\omega(E(Q)+\delta)} \delta.
\end{align*}
Finally, choosing $\delta >0$ small enough so that $\frac{\|
g'\|_{L^\infty}}{\omega(E(Q)+\delta)} \delta \le 2A$, we get
\[ \int \left( u_r^2 + \frac{d(u)}{r^2} \right) rdr  \ge 4A - \frac{\|
g'\|_{L^\infty}}{\omega(E(Q)+\delta)} \delta \ge 2A \ge \frac{A}{E(Q)}
E(u). \]
This gives the lower bound with constant $\frac{A}{E(Q)}$.

Assume now that $\| u \|_{L^\infty} \le D^*$. Then $d(u) \ge 0$. As
$f(x) \sim k^2 x$ as $x \to 0$, let $D >0$ be such that $|f| \ge k^2/2 x$
on the interval $[-D,D]$. If $\| u \|_{L^\infty} \le D$, then of course
\[ F(u) \ge \int u_r^2 rdr + \frac{k^2}{2\| g' \|_{L^\infty}^2}  \int
\frac{g^2(u)}{r^2} rdr \ge \min \left( 1, \frac{k^2}{2\| g'
\|_{L^\infty}^2} \right) E(u). \]
Otherwise, arguing as before, $\| u \|_{L^\infty} \in [D, D^*]$ and  we see
that $F(u) \ge 4 \int_0^D \sqrt{d}$ so that (as $E(u) < E(Q) + \delta \le
2 E(Q)$)
\[ F(u) \ge \frac{2 \int_0^D \sqrt{d}}{E(Q)} E(u). \]
Choosing $\delta >0$ small enough and $c =\min (2 (\int_0^D
\sqrt{d})/E(Q), A/E(Q) k^2/(2\| g'\|_{L^\infty}^2), 1)$ ends the proof.
\end{proof}

Let $\varphi$ be such that $\varphi(r) = 1$ if $r \le 1$, $\varphi(r) = 0$ if
$r \ge 2$, and $\varphi (r) \in [0,1]$. Denote $\varphi_R(x) =
\varphi(r/R)$.

In the notation $\q O$, constants are absolute (do not depend on $R$ or $t$
or $u$).

\begin{lem} \label{viriel}
Let $(u,u_t) \in \q V(\delta)$ be a solution to (\ref{wm}). One has
\begin{gather*}
\frac{d}{dt} \int u_t u_r r^2 \varphi_R(r) dr  = - \int u_{t}^2 rdr + \q
O(E_R^\infty(u,u_t)), \\
\frac{d}{dt} \int u u_t r \varphi_R(r) dr  = \int \left( u_t^2- u_r^2 -
\frac{uf(u)}{r^2} \right) r \varphi_R(r) dr + \q O(E_R^\infty(u,u_t)).
\end{gather*}
\end{lem}

\begin{nb}
For the $\q O$, we can consider the rest of the energy $E_R^\infty$ or
equivalently the tail in $H \times L^2$
\[ \tau(R,u,u_t) = \int_R^\infty \left( u_t^2 + u_r^2 + \frac{u^2}{r^2}
\right) rdr. \]
\end{nb}

\begin{proof}
One computes
\begin{align*}
\lefteqn{ \frac{d}{dt} \int u_t u_r r^2 \varphi_R(r) dr }\\
& = \int u_{tt} u_r r^2 \varphi_R(r) dr + \int u_{t} u_{rt} r^2
\varphi_R(r) dr \\ 
& = \int \left( u_{rr} + \frac 1 r u_r - \frac{f(u)}{r^2} \right) u_r r^2
\varphi_R(r) dr - \frac 1 2 \int u_{t}^2 (2r \varphi_R(r) + r^2
\varphi_R'(r)) dr \\
& = - \frac 1 2 \int u_{t}^2 (2r \varphi_R(r) + r^2
\varphi_R'(r)) dr + \int u_r^2 (r \varphi_R(r) -
\frac 1 2 (r^2\varphi_R(r))')) dr + \frac 1 2 \int g^2(u) \varphi_R'(r) dr \\
& =  -  \int u_{t}^2 r \varphi_R(r) dr +
\frac 1 2 \int \left( u_t^2 - u_r^2 +  \frac{g^2(u)}{r^2} \right) r^2
\varphi_R'(r) dr 
\end{align*}
Now, notice that 
\begin{align*}
\lefteqn{ \left| - \int u_{t}^2 r(1-\varphi_R(r)) dr -
\frac 1 2 \int \left( u_t^2 - u_r^2 +  \frac{g^2(u)}{r^2} \right) r^2
\varphi_R'(r) dr \right| } \\
& & & \le \int e(u,u_t) (1-\varphi_R(r)) rdr + \int e(u,u_t) r^2
|\varphi_R'(r)| dr \\ 
& & & \le  E_R^\infty(u,u_t) +  \frac 1 R \int e(u) r^2 |\varphi'(r/R)| dr \\
& & & \le  (1 + 2 \| \varphi' \|_{L^\infty})  E_R^\infty(u,u_t).
\end{align*}
>From this, we immediately deduce
\[ \frac{d}{dt} \int u_t u_r r^2 \varphi_R(r) dr = - \int
u_{t}^2 r dr + \q O (E_R^\infty(u,u_t)). \]
In the same way,
\begin{align*}
\frac{d}{dt} \int u u_t r \varphi_R(r) dr & = \int u_t^2 r \varphi_R(r) dr
+  \int u u_{tt} r \varphi_R(r) dr \\
& = \int u_t^2 r \varphi_R(r) dr + \int u \left(u_{rr} + \frac 1 r u_r -
\frac{f(u)}{r^2} \right) r \varphi_R(r) dr \\
& = \int \left( u_t^2- u_r^2 - \frac{uf(u)}{r^2} \right) r \varphi_R(r) dr
+  \frac 1 2 \int u^2  (r \varphi_R(r))'' dr \\
& \quad - \frac 1 2 \int u^2 \varphi_R'(r) dr.
\end{align*}
Then similarly
\begin{align*}
\lefteqn{ \left| \int \left( u_t^2- u_r^2 -
\frac{uf(u)}{r^2} \right) r (1-\varphi_R(r)) dr + \frac 1 2 \int u^2  (r
\varphi_R(r))'' dr - \frac 1 2 \int u^2 \varphi_R'(r) dr \right| } \\
& & & \le \int \left| u_t^2- u_r^2 -
\frac{uf(u)}{r^2} \right| (1-\varphi_R(r)) r dr + \frac 1 2 \int
\frac{u^2}{r^2} |r^2 \varphi_R''(r) + r \varphi'_R(r) | rdr \\
& & & \le C \int e(u,u_t) (1-\varphi_R(r)) r dr + C \int \frac{g^2(u)}{r^2}
\left| \frac{r^2}{R^2} \varphi''(r/R) - \frac{r}{R} \varphi'(r/R) \right|
rdr \\
& & & \le C E_R^\infty(u,u_t) + C(4 \| \varphi'' \|_{L^\infty} + 2 \varphi'
\|_{L^\infty}) E_R^\infty(u,u_t).
\end{align*}
(The bounds on the third line come respectively from the pointwise bounds
$|uf(u)| \le C g^2(u)$ and $u^2 \le Cg^2(u)$, which holds according to the
proof of Lemma \ref{varlem} (as $E(u) \le E(Q)$). 
\end{proof}

\begin{thm}[Rigidity property] \label{rigidity}
Let $(u_0,u_1) \in \q V(\delta)$, and denote by $u(t)$ the
associated solution. Suppose that for all $t \ge
0$, there exist $\lambda(t) \ge A_0 >0$ such that
\[ K = \left\{ \left. \left( u \left( t,\frac{r}{\lambda(t)} \right), \frac 1
{\lambda(t)} u_t\left(t,\frac{r}{\lambda(t)} \right) \right) \right| (t,r)
\in \m R_+ \right\} \text{ is precompact in } H \times L^2. \]
Then $u \equiv 0$.
\end{thm}

\begin{proof}
Recall that $u$ is global due to Proposition \ref{gwp}.

As $K$ is precompact and $\lambda(t) \ge A_0>0$, for all $\e>0$, there exists
$R(\e)$ such that
\[ \forall t \ge 0, \quad E_{R(\e)}^\infty(u,u_t) < \e. \]
This means that
\[ \lim_{R \to \infty} \sup_{t \ge 0} E_R^\infty(u,u_t) =0. \]
Due to Lemma \ref{viriel} and \ref{varlem}, we have
\begin{align*}
\frac d {dt} \left( \int u_t u_r r^2 \varphi_R(r)dr + \frac 1 2 \int uu_t
\varphi_R(r) dr \right) & = - \frac 1 2 \int \left( u_t^2 + u_r^2 +
\frac{uf(u)} {r^2} \right) rdr + \q O(E_R^\infty(u,u_t)) \\
& \le - \frac{1}{2C} E(u,u_t) + \q O(E_R^\infty(u,u_t)).
\end{align*}
Fix $R$ large enough so that $\sup_{t \ge 0}  \q O(E_R^\infty(u,u_t)) \le
\frac{E(u)}{4C}$. Then by integration between $\tau=0$ and $\tau= t$ and
conservation of energy~:
\[ \int u_t u_r  r^2 \varphi_R(r) dr + \frac 1 2 \int u
u_t r \varphi_R(r) dr \le - \frac{E(u,u_t)}{4C} t + C_0. \]
However, from finiteness of energy and $u^2 \le  C g^2(u)$, we have
for all $t$, 
\begin{align*}
\lefteqn{ \left| \int u_t u_r  r^2 \varphi_R(r) dr + \frac 1 2 \int u
u_t r \varphi_R(r) dr \right| } \\
& \le \frac 1 2 \int (u_t^2 + u_r^2) r^2
\varphi_R(r) dr + \frac 1 4 \int \left( u_t^2 + C \frac{g^2(u)}{r^2}
\right) r^2 \varphi_R(r) \\
& \le  R E(u,u_t) + \frac{C}{2}  R E(u,u_t),
\end{align*}
so that this quantity is bounded, hence $t \le 4C(2 + \sqrt C)R$. This is a
contradiction with the fact that $u$ is global in time. 
\end{proof}

\section{Proofs of Theorem \ref{th1} and Corollary \ref{cor1}}

\begin{proof}[Proof of Theorem \ref{th1}]
>From Theorem \ref{lwp}, we only need to show that  $\| u \|_{S(\m R)} <
\infty$.

We consider the critical energy (for $\delta$ as in Lemma \ref{varlem})
\[ E_c = \sup \{ E \in [0,E(Q)+\delta] | \ \forall (u_0,u_1) \in \q
V(\delta), \ E(u_0, u_1) < E \imp  \| u(t) \|_{S(\m R)} < \infty \}. \]
(recall that for initial data $(u_0,u_1) \in \q V(\delta)$, Proposition
\ref{gwp} already ensures that the corresponding wave map $u(t)$ is global in
time). 

Theorem \ref{th1} is the assertion $E_c = E(Q) + \delta$. Assume this is
not the case.

Notice that $E_c \ge \delta_0 >0$, due to Theorem \ref{lwp}.

>From the work of Bahouri and Gerard \cite{BG99}, the compensated
compactness procedure of Kenig and Merle in \cite{KM06b} provides us with a
critical element $(u^c,u^c_t)$ (in the case $E_c < E(Q) + \delta$)~:

\begin{prop}\label{critelement}
There exists $(u^c_0, u^c_1) \in H \times L^2$, satisfying $(u^c_0, u^c_1)
\in \q V(\delta)$, $E(u^c_0, u^c_1) = E_c$ and if we denote $(u^c,u^c_t)$
the associated solution to Problem (\ref{wm}), $u^c(t)$ is global and
$\| u^c \|_{S(\m R)} = \infty$.
\end{prop}
(Notice that $u_c$ is global due to the energy bound and Proposition
\ref{gwp}). 

We can assume without loss of generality that $\| u^c \|_{S(\m R^+)} =
\infty$. 
Following Kenig and Merle \cite{KM06b}, we also have (possibly changing $u^c$)
\begin{prop} \label{compactprop}
There exist $A_0 > 0$ and a continuous function $\lambda : \m R^+ \to [A_0,
\infty)$ such that the set
\[ K = \left\{ v(t) \in H \times L^2 | v(t,r) = \left( u^c \left(t,
    \frac{r}{\lambda(t)}\right), \frac{1}{\lambda(t)}
    u^c_t \left(t,\frac{r}{\lambda(t)} \right) \right) \right\} \]
has compact closure in $H \times L^2$.
\end{prop}

>From Theorem \ref{rigidity}, we deduce that $(u^c,u^c_t) =(0,0)$, which is a 
contradiction with $E(u^c,u_t^c) =E_c >0$. Hence $E_c = E(Q)+ \delta$.
\end{proof}

\begin{proof}[Proof of Corollary \ref{cor1}]
Notice that if $(u_0,u_1)$ is such that $E(u_0,u_1) \le E(Q)$ and 
$(u_0,u_1) \notin \q V(\delta)$, then (as $u_0(0) =0$), $u_0(\infty) \ge C^*$, 
and from the pointwise inequality (\ref{ptw}), $u_0(\infty) = C^*$, $u_0(r)
= \e Q(\lambda r)$ for some $\lambda > 0$ and $\e \in \{ -1,1\}$, and $u_1 =0$.

Hence in our case, $(u_0,u_1) \in \q V(\delta)$, and the result follows from 
Theorem \ref{th1}.
\end{proof}

\nocite{KM06a}
\nocite{Ger98}

\bibliographystyle{amsplain}
\bibliography{../../references}

Raphaël Côte\par
Centre de Mathématiques Laurent Schwartz\par
École polytechnique\par
91128 Palaiseau Cedex\par
FRANCE\par
\texttt{cote@math.polytechnique.fr}

\bigskip

Carlos E. Kenig\par
Department of Mathematics\par
University of Chicago\par
5734 University Avenue\par
Chicago, IL 60637-1514\par
USA\par
\texttt{cek@math.uchicago.edu}

\bigskip

Frank Merle\par
Département de Mathématiques\par
Université de Cergy-Pontoise / Saint-Martin\par
2, avenue Adolphe Chauvin\par
95 302 Cergy-Pontoise Cedex\par
FRANCE\par
\texttt{frank.merle@u-cergy.fr}

\end{document}